\documentclass[12pt,oneside]{amsart}
\usepackage{amssymb, amscd, amsmath, amsthm, latexsym, hyperref}
\usepackage{pb-diagram, fancyhdr, psfrag}
 \usepackage[text={6.3in,8.5in},centering,letterpaper,dvips]{geometry}
 
 \usepackage[dvips]{graphicx} 
 \usepackage{color}
\newcommand{\compactlist}{\begin{list}{$\bullet$}{\setlength{\leftmargin}{1em}}}
\def\zz{{\mathbb Z}}

\def\qq{{\mathbb Q}}

\def\rr{{\mathbb R}}
\def\co{\colon\thinspace}

\def\cala{\mathcal{A}}
\def\calc{\mathcal{C}}
\def\calh{\mathcal{H}}

\def\cs{\mathbin{\#}}

\newcommand{\spinc}{\ifmmode{{\mathfrak s}}\else{${\mathfrak s}$\ }\fi}
\newcommand{\spinct}{\ifmmode{{\mathfrak t}}\else{${\mathfrak t}$\ }\fi}

\newcommand{\fig}[2] { \includegraphics[scale=#1]{#2} }
\def\U{\Upsilon}
\DeclareMathOperator{\Wh}{Wh}

\newtheorem{theorem}{Theorem}
\newtheorem{lemma}[theorem]{Lemma}
\newtheorem{corollary}[theorem]{Corollary}
\newtheorem{question}[theorem]{Question}

\theoremstyle{definition}

\newtheorem{example}[theorem]{Example}

\begin{document}
\title{Knot concordances and alternating knots}
\author{Stefan Friedl,  Charles Livingston and Raphael Zentner}
\thanks{ The first and the third authors were  partially supported by the SFB 1085 `higher invariants' at the University of Regensburg.  The second author was  partially supported by a grant from the Simons Foundation and by NSF-DMS-1505586. }

\address{Stefan Friedl: Fakult\"at f\"ur Mathematik, Universit\"at Regensburg, Germany}
\email{sfriedl@gmail.com}

\address{Charles Livingston: Department of Mathematics, Indiana University, Bloomington, IN 47405 }
\email{livingst@indiana.edu}

\address{Raphael Zentner: Fakult\"at f\"ur Mathematik, Universit\"at Regensburg, Germany}
\email{raphael.zentner@mathematik.uni-regensburg.de}


\begin{abstract} There is an infinitely generated free subgroup   of the smooth knot concordance group with the property that no nontrivial element in this subgroup can be represented by an alternating knot.  This subgroup has the further property that every element is represented by a topologically slice knot.

\end{abstract}

\maketitle

\section{Introduction}

Let $\calc_a \subset \calc$ denote the subgroup of the smooth three-dimensional knot concordance group  generated by the set of all alternating knots.  A simple construction reveals that connected sums of alternating knots, as well as their mirror images, are alternating; it follows that every element in $\calc_a$ is represented by an alternating knot.

It has been known that $\calc_a \not= \calc$; in Section~\ref{section:algebraic} we  prove the following result, which also holds in the topological category. 

\begin{theorem}\label{theorem:mainalg}
For every positive integer $N$, there is an infinitely generated free subgroup $\calh_N\subset \calc/\calc_a$ with the following property: 
 If $K$ represents a nontrivial class in $\calh_N$, then any cobordism from $K$ to an alternating knot has genus at least $N$.
 
\end{theorem} \vskip.05in

Our main result  is this generalization.
\begin{theorem}\label{theorem:main}
For every positive integer $N$, there is an infinitely generated free subgroup $\calh_N\subset \calc/\calc_a$ with the following properties. 
\begin{enumerate}
\item If $K$ represents a nontrivial class in $\calh_N$, then any cobordism from $K$ to an alternating knot has genus at least $N$.
 
\item Every class in $\calh_N$ is represented by a topologically slice knot; in particular,  the quotient $\calc_{ts}/(\calc_{ts}\cap \calc_a)$ is infinitely generated, where $\calc_{ts}$ is the concordance group of topologically slice knots.
\end{enumerate} 
\end{theorem} \vskip.05in

These results were inspired by three 
results concerning  alternating knots.  In~\cite{oss}, Ozsv\'ath-Stipsicz-Szab\'o showed that the concordance invariant $\U_K(t)$, a piecewise linear function on $[0,2]$, provides an obstruction to a knot being concordant to an alternating knot.  In an earlier paper, Abe~\cite{abe} showed  that the difference between the Rasmussen invariant and the signature of a knot provides a bound on the    {\it alternation number} of knot, alt($K$), the minimum number of crossing changes required to convert $K$ into an alternating knot. Extending these results, Feller-Pohlmann-Zentner~\cite{fpz} used     $\U_K(t)$ to find another  lower bound on the alternation number. 
 The work here is built from the  key observations of~\cite{abe, fpz}.

 For a knot $K$, let $\cala_g(K)$ denote the minimum genus of a cobordism from $K$ to an alternating knot and   let $\cala_s(K)$ denote the  minimum number of double point singularities in a generically immersed   concordance from $K$ to an alternating knot. 
It is  straightforward  to show that both $\cala_g$ and $\cala_s$ induce functions from  $ \calc/\calc_a $ to the nonnegative integers.

The focus of~\cite{abe, fpz} is on bounding  alt($K$).  Notice that   
a  sequence of crossing changes on a knot yields a singular concordance, and its singularities can be resolved to produce an embedded surface. Thus, a lower  bound on $\cala_g(K)$ yields a bound on $\cala_s(K)$ and hence on alt($K$).    The approach used in~\cite{abe, fpz} leads to much better bounds on alt$(K)$ and $\cala_s(K)$ than can be obtained via those on $\cala_g(K)$.  These will be presented in Section~\ref{section:singular}, after we conclude our work on the bounds on $\cala_g$. 
 
 \vskip.05in
 
In the context of Heegaard Floer homology, alternating knots form a subgroup of the set of {\it quasi-alternating knots}.  Using the results
 of~\cite{man-oz}, all our results extend to this setting; our theorems could be stated in terms of $\calc_{qa}$, the subgroup of the concordance group generated by quasi-alternating knots, which, by \cite[Lemma~2.3]{champanerkar-kofman}, consists precisely of the quasi-alternating knots.

\vskip.05in
\noindent{\it Acknowledgments}  Thanks go to Peter Feller for his helpful comments and his careful reading of a preliminary version of this paper. 

\section{Algebraic invariants}\label{section:algebraic}

We begin with an elementary result.

\begin{lemma}\label{theorem:alggenusbound}
Let $\nu \co \calc \to \rr$ be a homomorphism that vanishes on $\calc_a$ and has the property that for all knots $K$, $g_4(K) \ge |\nu(K)|$.    
The genus of any cobordism from a knot $K$ to an alternating knot is greater than or equal to $|\nu(K)|$; that is, for all $K$, $\cala_g(K) \ge |\nu(K)|$.
\end{lemma}

\begin{proof} 

 Suppose that there is a genus $g$ cobordism from $K$ to an alternating knot $J$.  Then $K \cs -J$ satisfies $g_4(K \cs -J) \le g$.  Thus, $|\nu (K \cs -J)| \le g$.  Since $\nu$ is a homomorphism and $\nu(J) = 0$, this gives $|\nu(K)| \le g$, as desired.

\end{proof}

Let $J_\eta(K)$ denote the jump in the Levine-Tristram signature function $\sigma_K$
at a point $\eta$ on the unit circle,~\cite{tristram}. For all $\omega$ that are not roots of the Alexander polynomial of $K$,  $|\sigma_K(\omega)| \le 2g_4(K)$. Thus, for all $\eta$,  $|J_\eta(K)| \le 4g_4(K)$.  (According to~\cite{matumoto}, this jump function equals the Milnor signature~\cite{milnor} at $\eta$.)

 To apply these jump functions to study alternating knots, we begin by considering a family of Alexander polynomials whose roots have special properties.

\begin{lemma} For $n\ge 1$, the polynomial $$\Delta_n(t) = t^{4} + n t^3 -(2n+1) t^2+ nt + 1$$ is irreducible in $\qq[t]$. It has a unique root, $\omega_n$, with positive imaginary part, and $\omega_n$ lies on the unit circle.  It has two real roots, both of which are negative.
\end{lemma}

\begin{proof}
By Gauss's Lemma, we only need to show that $\Delta_n(t)$ is irreducible in $\zz[t]$.  The only possible linear factors could be $t \pm 1$; these are ruled out by the condition that  $\Delta_n(1) = 1$.  The only possible quadratic factors are of the form $t^2  +at \pm 1$. The condition  $\Delta_n(1) = 1$ constrains $a$ to be in a finite set, and one can check the few possibilities.

We can rewrite  $\Delta_n (t)$ as
$$\Delta_{n}(t) = t^2 \left( (t + t^{-1})^2 + n(t + t^{-1})  - (2n+3) \right).$$ 

By applying the quadratic formula, we see that  roots are solutions to  
$$t + \frac{1}{t} =  \frac{ -n \pm \sqrt{n^2 + 8n+12} }{2} . $$  In the case that the right hand side is negative,  it is clearly less than $-2$. The maximum negative value of $t +\frac{1}{t}$ is $-2$,  and from this one can show that there exist two  real solutions for $t$, both negative.  In the case that the right hand side is positive, a brief calculation shows that it is less than 2. Since the minimum positive value of  $t +\frac{1}{t}$ is $2$,  there are no corresponding real roots.

The symmetry of the polynomial implies that if $\alpha$ is root with positive imaginary part, then so is $\overline{\alpha^{-1}}$.  There is exactly one root $\alpha$ with positive imaginary part, so it must be that $\alpha = \overline{\alpha^{-1}}$ and hence, $\alpha$ lies on the unit circle.

\end{proof}

\begin{theorem} For all $n>0$, the homomorphism $J_{\omega_n}$ vanishes on $\calc_a$.
\end{theorem}

\begin{proof}
We use the theorem of Murasugi~\cite{murasugi0} stating that alternating knots have Alexander polynomials of the form $\sum_{k=0}^N (-1)^k a_k t^k$ for some $N\ge 0$ and all $a_k >0.$ (See also,~\cite{crowell, murasugi1}.) In particular, for an alternating knot $K$, all real roots of the Alexander polynomial are positive.

Suppose that $J_{\omega_n}(K) \neq 0$.  Nontrivial jumps in the signature function can occur only at roots of the Alexander polynomial, hence $\Delta_K(\omega_n) = 0$.  Since $\Delta_n(t)$ is irreducible with root $\omega_n$, it follows that $\Delta_n(t)$ divides $\Delta_K(t)$, and thus $\Delta_K(\alpha) = 0$, where $\alpha$ is one of the negative real  roots of $\Delta_n(t)$.  Thus, $\Delta_K(t)$ has a negative root, so $K$ is not alternating. 
\end{proof}

By a theorem of Seifert~\cite{seifert}, there exists a knot with Alexander polynomial $\Delta_n(t)$; denote some fixed choice of such a knot by  $K_n$.
In Figure~\ref{figure:Kn} we illustrate one possible choice for $K_n$.  The reader may easily find a Seifert matrix and compute the Alexander polynomial for this knot.  Notice that because one band in the Seifert surface is unknotted and untwisted, the four-genus is at most one; since the Alexander polynomial is irreducible, it is of four-genus exactly one.

\begin{figure}[h]
\fig{.15}{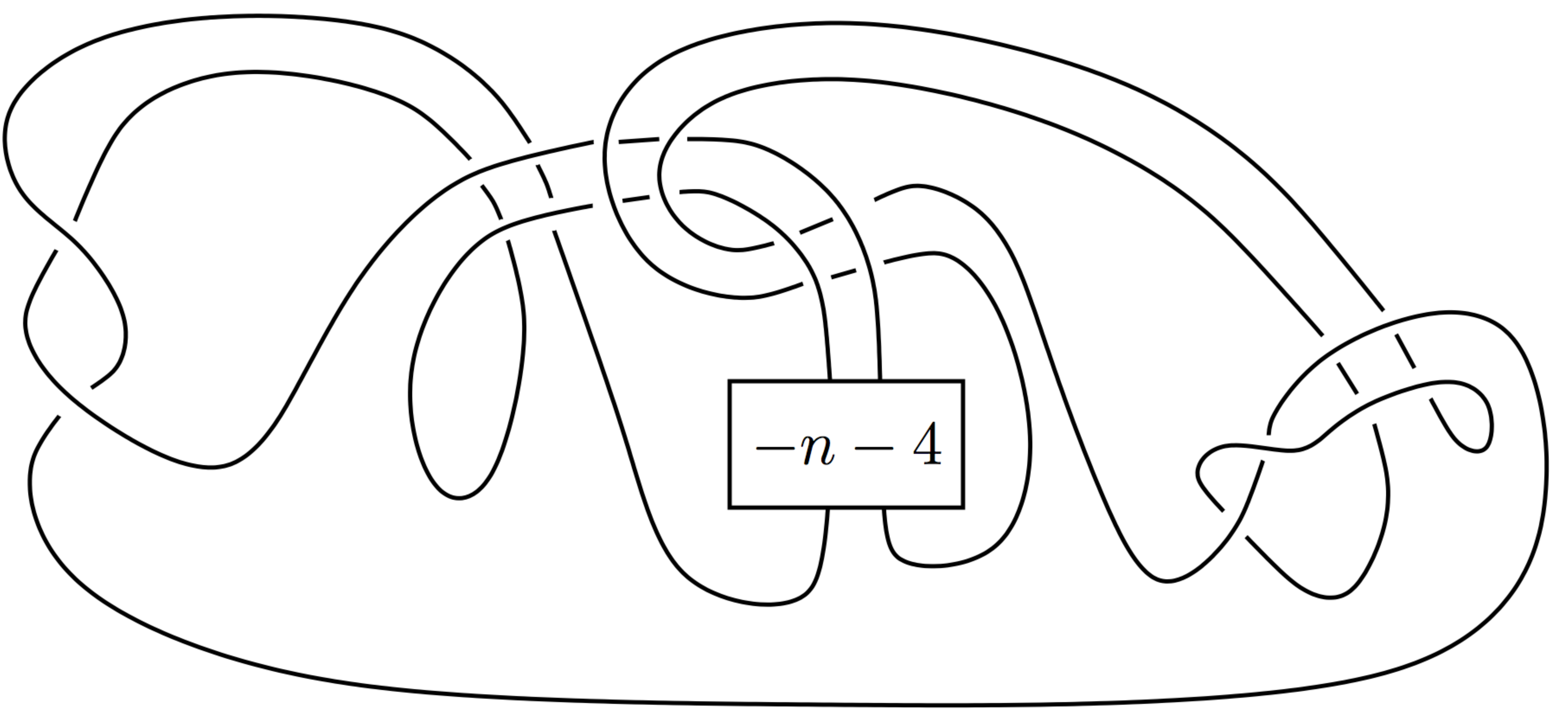}
\caption{$K_n$}\label{figure:Kn}
\end{figure}

We can now prove Theorem~\ref{theorem:mainalg}, which we state in more detail.

\begin{theorem} \label{theorem:mainalg2}  The set of knots $\{K_n\}_{n>0}$ is independent in $\calc / \calc_a$.  For any integer $N >0$, if $K$ is a non-trivial knot in the span of $\{2N K_n\}$,
then any cobordism from $K$ to an alternating knot has genus at least $N$.

\end{theorem}

\begin{proof} Write $K = \sum_{i=1}^{M} a_i 2N K_i$ with some $a_i \neq 0$ and $M\geq 1$.  Let $n$ be any value of $i$ for which $a_i \neq 0$.

 Over the real numbers, $\Delta_n(t) $ factors irreducibly as $$\Delta_n(t) = (t- \beta_1)(t-\beta_2)(t^2 -2\cos(\theta_n)t +1), $$ where $\omega_n = \cos(\theta_n) + i \sin(\theta_n)$.  According to Milnor~\cite{milnor},  since $t^2 - 2\cos(\theta_n)t +1$ is a factor of multiplicity one in $\Delta_{K_n}(t)=\Delta_n(t)$, it follows that $J_{\omega_n}(K_n) = \pm 2$.   We now have $J_{\omega_n}(K) = \pm 4a_n N$.  Recalling that the jump function is bounded above by $4|g_4(K)|$,  Theorem~\ref{theorem:alggenusbound}, implies that $g_4(K) \ge |a_n N| \ge N$.

\end{proof}


\section{Homomorphisms on $\calc$ and alternating knots}\label{} 

Results such as those in the previous section cannot be applied to topologically slice knots; all signature invariants vanish.  Thus, we are led to consider a new family of knot invariants. 

\begin{theorem}\label{theorem:bound}  Suppose that $\nu_1$ and $\nu_2$ are real-valued homomorphisms on $\calc$ that satisfy $|\nu_i(K)| \le g_4(K)$ for all $K$ and for which $\nu_1(K) = \nu_2(K)$ for all alternating knots $K$.  Then the homomorphism $\psi_{1,2}  = \nu_1 - \nu_2$ induces a homomorphism on $\calc/\calc_a$, and for all knots $K$, 
$$\cala_g(K) \ge \frac{1}{2}|\psi_{1,2}(K)|  .$$
 \end{theorem}

\begin{proof}
This is an immediate consequence of Theorem~\ref{theorem:alggenusbound}.
\end{proof}

\begin{lemma}\label{lemma:translate} The following homomorphisms defined on $\calc$ bound the four-genus and are equal on alternating knots.
\begin{itemize}

\item The quotient of the classical knot signature,  $\sigma(K)/2$.\vskip.05in

\item The negative Ozsv\'ath-Szab\'o invariant, $-\tau(K)$.\vskip.05in

\item The quotient of the Rasmussen invariant,  $s(K)/2$. \vskip.05in

\item The Upsilon function $\U_K(t)/t$, for each $t \in (0,1]$.  \vskip.05in

\item The ``little upsilon function,'' 
$\upsilon(K) = \U_K(1)$.\vskip.05in

\end{itemize}
\end{lemma}

\begin{proof}  The fact that the first four of these invariants bound the four-genus is proved in the  original references,~\cite{murasugi, os1, oss, rasmussen}.  Futhermore, in the references~\cite{ os1, oss, rasmussen} it is shown that each of $-\tau(K)$, $\U_K(t)/t$, and $s(K)/2$, agrees with $\sigma(K)/2$ for alternating knots $K$.  The last, $\upsilon(K)$, is a specialization of $\U_K(t)/t$, used in studying alternating knots in~\cite{fpz}.
\end{proof}

\section{Basic Examples}

\subsection{Examples based on  $\sigma$, $\tau$ and $s$.}  

\begin{example} One can compute that $\sigma(T(3,7))/2 = -4$ and $-\tau(T(3,7)) = -6$; hence $|\sigma(T(3,7))/2 - (-\tau(T(3,7))) |= 2$.  Thus, by Theorem~\ref{theorem:bound}  any cobordism from $T(3,7)$ to an alternating knot must have genus at least one.
\end{example}

\begin{example}
Similarly, according to Hedden and Ording~\cite{hedden-ording}, the values of $-\tau$ and $s/2$ differ by one on the twice-twisted positive Whitehead double of the trefoil knot,  $\Wh^+(T(2,3),2)$, and thus this knot is not concordant to an alternating knot.  Here $\Wh^+(\ \cdot\ ,2)$ denotes the positive clasped, twice twisted, Whitehead double and $C_{n,2n-1}(\ \cdot\ )$ denotes the $(n,2n-1)$--cable.
\end{example}

Notice that $-\tau$ and $s/2$ agree on $T(3,7)$ and so we find that  $T(3,7)$ and  $\Wh^+(T(2,3),2)$ generate a   rank two free subgroup of $\calc/ \calc_a$:  the pair of homomorphisms $\sigma(\cdot)/2 - \tau(\cdot)$ and $\tau(\cdot) + s(\cdot)/2$ define an injection from the span of these two knots in $\calc/\calc_a$ to $\zz \oplus \zz$.

\begin{example}

One can use $\tau$ and $\sigma$ to give examples of  topologically slice knots that are nontrivial in $\calc / \calc_a$.  According to~\cite{livingston1}, the untwisted Whitehead double of the trefoil knot, 
$\Wh^+(T(2,3),0)$, has $\tau = 1$  and $\sigma = 0$.  Thus, it is not concordant to an alternating knot.
\end{example}

\subsection{Examples based on the Upsilion invariant}

For any $s, t \in (0,1]$, we let $\psi_{s,t} (K)= \U_K(s)/s - \U_K(t)/t$.  There is an  immediate consequence of Theorem~\ref{theorem:bound}.

\begin{theorem}\label{theorem:genusbound} For all knots $K$,
$\cala_g(K) \ge \frac{1}{2} | \psi_{t,s}(K)| $.
\end{theorem}

\begin{example}
Consider the difference $K = T(3,7) \# - T(4,5)$.   Figure~\ref{figure:3745}  illustrates the function $\U_K(t)/t$.  Now let $\psi_{t,s}(K)  =\U_K(s)/s - \U_K(t)/t$. As the graph indicates, $|\psi_{0,2/3}(K)| = 1$.  Thus, any cobordism  from $K$ to an alternating knot must have genus at least 1.  This could not have been determined using $\tau(K) =0$ or $\upsilon(K) = 0$.  Similarly, separate computations show that  both the Rasmussen invariant and the signature vanish.  Thus, to detect this knot, we must use $\U_K(t)$ for some $t \in (0,1)$.  

\begin{figure}
\fig{.6}{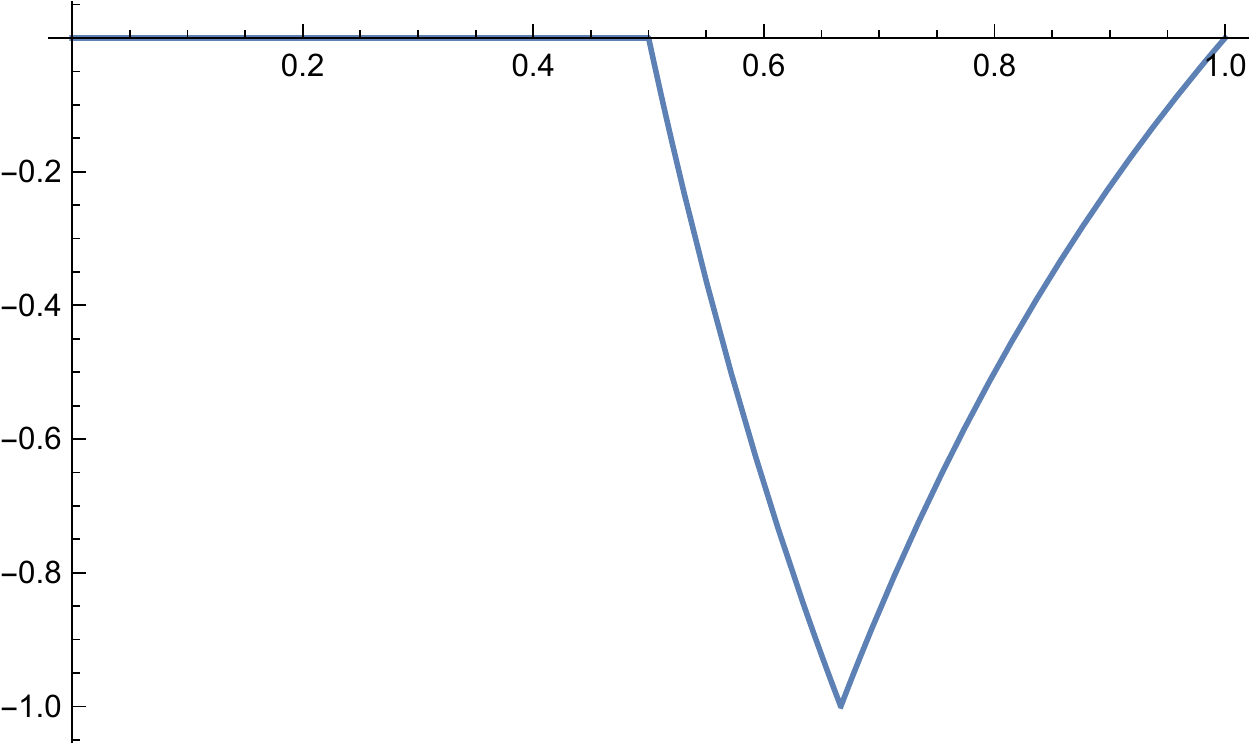}
\caption{$\U_{T(3,7) - T(4,5)}(t)/t$}\label{figure:3745}
\end{figure}
\end{example}

\subsection{An infinite independent family in $\calc / \calc_a$.}

By using different values of $s$ and $t$, we prove the following result.
 
\begin{theorem} \label{theorem:torus} $(1)$  The set of knots $\{T(p,p+1)\}$,  $p\ge 3$, is linearly independent in $\calc/\calc_a$.  $(2)$ For every integer $N >0$,  there exists a sequence of positive integers $\{N_p\}$ such that the set  $S = \{N_pT(p,p+1)\}$ has the property that every non-trivial knot in span$(S)$ satisfies $\cala_g(J) \ge N$.
\end{theorem}

\begin{proof}

To simplify notation, let $K_p = T(p,p+1)$.
 Using the methods of~\cite{oss}, we have that the first singularity of $\U_{K_p}(t)$ is at $t = 2/p$ and the second singularity is at $4/p$.  

For each $p$, choose an $\epsilon_p$ with $2/p < 2/p + \epsilon_p < 2/(p-1)$. Consider the functions $\phi_{2/p , 2/p + \epsilon_p}$, which we abbreviate $\psi_p$ for the moment.  Then $\psi_p(K_p)\ne 0$ and $\psi_p(K_n) = 0$ for all $n <p$.  (We use here the fact that $\U_{K}(0) = 0$ for all knots $K$, so $\U_{K_p}(t) = c_pt$ on the interval $t \in [0,2/p]$ for some $c_p$.)

Suppose that some finite linear combination is trivial:  $\sum a_nK_n =0  \in \calc_a$.  Let $p$ be the largest value of $n$ for which $a_n \ne 0$.  Applying $\psi_p$ to the sum yields a nonzero multiple of $a_n$.  But $\psi_p$ would vanish if the sum were in $\calc_a$, implying that $a_p = 0$, a contradiction.  This completes the proof of linear independence.

For   statement (2), for each  $p$ select any integer $N_p$ so that    $\psi_p(N_pT(p,p+1)) \ge 2N$.   Following the same argument as in  (1), let $J_p = N_pT(p,p+1)$.   If $J = \sum a_nJ_n =0  \in \calc_a$, let $p$ be the largest $n$ for which $a_n \ne 0$.  Then $|\psi_p(J) |= |\psi_p(J_p)|  \ge 2N$.

\end{proof}

\section{Proof of Theorem~\ref{theorem:main}; an infinite family of topologically slice knots.}

We use the examples developed in~\cite{oss}.  For $n \ge 2$, let  
$$K_n = {\rm C}_{n,2n-1}(\Wh^+(T(2,3),0)) \#(-T(n,2n-1)).$$
Here  ${\rm C}_{n,2n-1}(\ \cdot\ )$ denotes the $(n,2n-1)$--cable.
This knot is topologically slice; to see this, observe that  the untwisted double is topologically slice, so its $(n,2n-1)$--cable is topologically concordant to $T(n,2n-1)$.   The proof of  ~\cite[Theorem 1.20]{oss} includes a computation of certain values of $\U_{K_n}(t)$. In summary:

\begin{theorem}\label{theorem:comp} For each $n>0$ there is an $\epsilon_n>0$ such that 

$$ \begin{cases}\U_{K_n}(t)
=0 &\mbox{if\  } t\le \frac{2}{2n-1}\\
\U_{K_n}(t)>0 &\mbox{if\  } \frac{2}{2n-1} <t < \frac{2}{2n-1} + \epsilon_n.
\end{cases}$$
\end{theorem}

To simplify notation, we let $a_n = \frac{2}{2n-1}$ and $b_n = a_n + \epsilon_n$.  Each $\epsilon_n$ can be chosen so that $b_{n+1} < a_n$. 

We now let  $\psi_n = \phi_{a_n, b_n}$.  As an immediate corollary to Theorem~\ref{theorem:comp} we have the following.
\begin{corollary}\label{corollary:calc} For all $n>0$, $\psi_n(K_n) >0$ and $\psi_n(K_m) = 0$ for $m<n$.
\end{corollary}

\subsection{Proof of Theorem~\ref{theorem:main}} The proof of Theorem~\ref{theorem:main}  is identical to that given above for Theorem~\ref{theorem:torus}, with the family of knots $K_n$ replacing the torus knots $T(p,p+1)$, and using   Corollary~\ref{corollary:calc} instead of the simpler result concerning the values of $\U_{T(p,p+1)}(t)$.   The subgroup $\calh_N$ is now generated by the set of knots $\{N_nK_n\}$ for appropriately chosen $N_n$.


\section{Singular concordances}\label{section:singular}

In this section we will consider the count of double points in singular concordances between knots.  In the case that one of the knots is the unknot, this is a well studied invariant; references include~\cite{cochran-gompf, kawamura, owens-strle}.

It is known that each of the  knot invariants $\sigma(K), -\tau(K), s(K)/2,$ and $\U_K(t)/t$, remains unchanged or increases by one if a positive crossing is changed to be a negative crossing.  This fact is implied by the following stronger theorem.  The proof is related to the proof of the crossing change bounds for $\tau$ and $\U_K(t)/t$ given in~\cite{livingston1, livingston2}.  We give the argument for more general invariants.

\begin{lemma}\label{theorem:one-sided} Let $\nu$ be a homomorphism from $\calc$ to $\rr$ satisfying
 \begin{enumerate} \item For all $K$, $|\nu(K)| \le g_4(K)$, \mbox{and}\vskip.05in
 \item $\nu(J) = -1$ for some knot $J$ with the property that changing a single crossing in $J$ from positive to negative yields a  slice knot.
 \end{enumerate}  Suppose there is a singular concordance from  $K_1$ and $K_2$ with precisely $s=s_+ + s_-$ double points, where $s_+$ is the number of  positive double points.   Then  $$ - s_+\le \nu(K_1) -\nu(K_2) \le s_-.$$

\end{lemma}

\begin{proof} 
  We show that for a knot $K$ bounding a singular disk with $s_+$ positive double points and $s_-$ negative double points,
 $$- s_+\le \nu(K)  \le  s_-.$$ This can be applied to $K = K_1 \# - K_2$ to complete the proof of the theorem.
 
We discuss the case of $s_- \le s_+$.  The singular disk can be converted into an embedded surface by tubing together $s_-$ pairs of positive and negative crossing points, and then resolving the remaining $s_+ - s_-$ double points.  This yields a surface bounded by $K$ of genus $s_+$.  Thus,
$  g_4(K) \le s_+$, so we also have $|\nu(K)| \le s_+$, and in particular,  $-s_+ \le \nu(K)$.

Notice that $K \# -(s_+ - s_-)J$ bounds a singular disk with $s_+$ negative double points as well as $s_+$ positive double points.  By tubing the double points together, we have 
$g_4(K \# -(s_+ - s_-)J) \le s_+$.  It follows that 
$$|\nu(K \# -(s_+ - s_-)J) | \le s_+.$$  Using the fact that $\nu(J) = -1$, this gives
$$| \nu(K) +s_+ - s_- |\le s_+,$$ and in particular, $\nu(K) \le s_-$.  This completes the proof in the first case, $s_- \le s_+$.  The second case can be proved similarly, or one can apply the first case to the knot $-K$.

\end{proof}

\noindent{\bf Note.} For each of the knot invariants we are considering, the knot $J$ can be taken to be $T(2,3)$.

The following bound on $\cala_s(K)$ is similar to the one given on $\cala_g$ in Theorem~\ref{theorem:genusbound}; note, however,  that the bound has been doubled. 

\begin{theorem}  Let $\nu_1$ and $\nu_2$ be any two homomorphisms on $\calc$ that satisfy the two conditions of Theorem~\ref{theorem:one-sided} and which agree on alternating knots. Then  the  homomorphism $\psi_{1,2}  = \nu_1 - \nu_2$ induces a homomorphism on $\calc/\calc_a$, and for all knots $K$, 
$$\cala_s(K) \ge |\psi_{1,2}(K)|  .$$
 \end{theorem}

\begin{proof}  Suppose that there is a singular concordance from $K$ to an alternating knot $K'$ with $s_+$ positive double points and $s_-$ negative double points.  Let $s = s_+ +s_-$. By Theorem~\ref{theorem:one-sided}, 
$$- s_+ \le \nu_1(K) -\nu_1(K') \le s_-$$
and
$$- s_- \le \nu_2(K') -\nu_2(K) \le s_+.$$
Adding these and using the fact that $\nu_1(K') = \nu_2(K')$ shows 
$$
-s \le \nu_1(K) - \nu_2(K) \le s,$$
which can be written in terms of $\psi_{1,2}$ as $|\psi_{1,2}(K) | \le s$.  This hold for all possible $s$ (including its minimum value) so  $|\psi_{1,2}(K) | \le \cala_s(K)$ as desired.

\end{proof}

\begin{example}
Consider $J = T(3,7) \# - T(2,11)$.  An illustration of $\U_J(t)$ is given in Figure~\ref{figure:37211}.
\begin{figure}
\fig{.6}{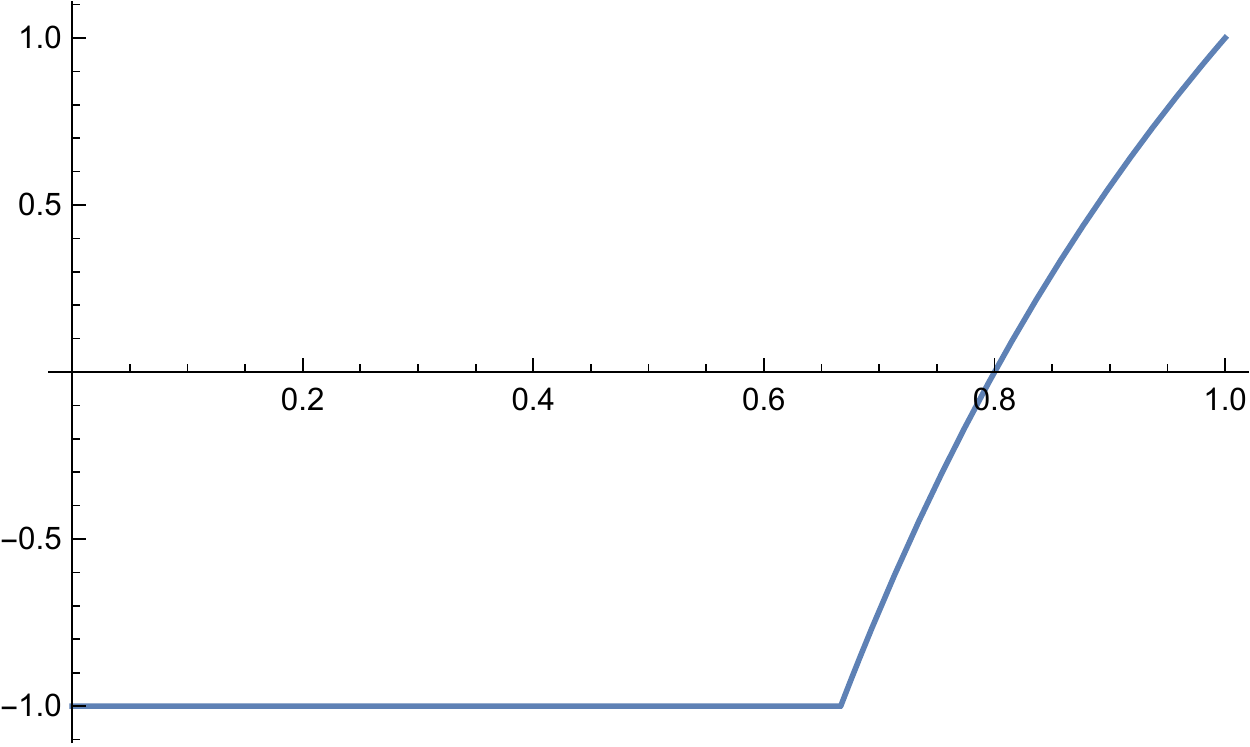}
\caption{$\U_{T(3,7) \# - T(2,11)}(t)/t$}\label{figure:37211}
\end{figure}
For this knot, we compute the  difference $|\upsilon(J) - \tau(J)| = 2.$  Thus, any singular concordance from $J$ to an alternating knot must have at least two singular points.  On the other hand, a result of Feller~\cite{feller0} implies that this knot has four-genus one, and hence there is a genus one cobordism from $J$ to an alternating knot.  Two crossing changes convert $T(3,7)$ into an alternating knot (see, for instance~\cite{kanenobu}), so $\cala_s(J) \le 2$.

In summary, for the knot $J = T(2,11) - T(3,7)$, $\cala_g(J) =1$ and $\cala_s(J) = 2.$   That is, both bounds from Theorem~\ref{theorem:main} are realized.
\end{example}

There is an immediate theorem, parallel to Theorem~\ref{theorem:main}.

\begin{theorem}\label{theorem:singular}
For every positive integer $N$, there is an infinitely generated free subgroup $\calh_N\subset \calc/\calc_a$ with the following properties. 
\begin{enumerate}
\item If $K$ represents a nontrivial class in $\calh_N$, then any generic singular concordance from $K$ to an alternating knot has at least $N$ double points.
 
\item Every class in $\calh_N$ is represented by a topologically slice knot; in particular,  the quotient $\calc_{ts}/(\calc_{ts}\cap \calc_a)$ is infinitely generated, where $\calc_{ts}$ is the concordance group of topologically slice knots.
\end{enumerate} 
\end{theorem} \vskip.05in

\section{Questions}

It has been shown that the quotient $\calc_{ts}/(\calc_{ts}\cap \calc_a)$ is infinitely generated; informally,  $\calc_{ts}\cap \calc_a$ is small compared to $\calc_{ts}$.  It is possible that $\calc_{ts}\cap \calc_a = 0$.  Thus, one can ask the following question.  

\begin{question}\label{question1} Is there an alternating knot that is topologically slice  but not smoothly slice?\end{question}

Initially, the only known examples of topologically slice knots that are not smoothly slice arose directly from Freedman's theorem~\cite{freedman} that knots $K$ with $\Delta_K(t) = 1$ are topologically slice. Non-trivial knots with trivial Alexander polynomial are not alternating, so this result is probably not  of immediate use in resolving Question~\ref{question1}.  Freedman's theorem was generalized in~\cite{friedl-teichner} to a class of knots with nontrivial alternating Alexander polynomials, offering a possible source of examples, and in~\cite{hlr, hkl} infinite families of slice knots which are not even concordant to polynomial one knots were constructed.     Unfortunately, it is not clear how to use those results to build alternating examples.

\vskip.05in
In a different direction, Litherland~\cite{litherland} proved that the set of positive torus knots is linearly independent in $\calc$.  Clearly, they become dependent in $\calc/\calc_a$, since the torus knots $T(2,2n+1)$ are alternating. 

\begin{question} Is the set of torus knots $\{T(p,q)\}$ with $3 \le p < q$ independent in $\calc / \calc_a$?\end{question}


\end{document}